\documentclass[11pt]{amsart}
\usepackage{latexsym,amsfonts,amssymb,graphicx,amsmath}
\usepackage{epstopdf}
\usepackage{amssymb}
\usepackage{enumerate}
\usepackage{enumitem}
\usepackage{booktabs}
\usepackage{stmaryrd}
\usepackage{sidecap}
\usepackage{color,soul}
\usepackage[font=small,labelfont=bf]{caption}
\numberwithin{equation}{section}

\newcommand{\remove}[1]{}

\newcommand{\Forall}{\;\forall\,}
\newcommand{\tmin}{\Theta_{\min}}

\newcommand{\Eh}{\mathcal{E}_h}
\newcommand{\Sh}{\mathcal{S}_h}

\newcommand{\ra}[1]{\renewcommand{\arraystretch}{#1}}

\newcommand{\dive}{\operatorname{div}}

\newcommand{\Th}{\mathcal{T}_h}

\newtheorem{theorem}{Theorem}
\newtheorem{lemma}{Lemma}
\newtheorem{proposition}{Proposition}

\theoremstyle{definition}

\newtheorem{Def}{Definition}[section]

\begin{document}
\title[Inf-Sup stability of unfitted Stokes  elements]{The Scott-Vogelius finite elements revisited}


\author[J. Guzm\'an]{Johnny Guzm\'an\textsuperscript{\textdagger}}
\address{\textsuperscript{\textdagger} Division of Applied Mathematics, Brown University, Providence, RI 02912, USA}
\email{johnny\_guzman@brown.edu}
\thanks{}

\author[R. Scott]{L. Ridgway Scott\textsuperscript{\textdaggerdbl}}
\address{\textsuperscript{\textdaggerdbl}Departments of  Computer Science and Mathematics,
Committee on Computational and Applied Mathematics,
University of Chicago, Chicago IL 60637, USA}
\email{ridg@uchicago.edu}
\thanks{}

\maketitle

\begin{abstract}  
We prove that the Scott-Vogelius finite elements are inf-sup stable on 
shape-regular meshes for piecewise quartic velocity fields and higher ($k \ge 4$). 

\end{abstract}
\medskip

\smallskip

\date{}

\section{Introduction}

In 1985 Scott and Vogelius \cite{scott1985norm} (see also \cite{vogelius1983right}) 
presented a family of finite element spaces in two dimensions which
when applied to the Stokes problem, produce velocity approximations that 
are exactly divergence free.  In addition, they proved that the method is stable by proving the pair of spaces satisfy the  so-called {\it  inf-sup}  condition.   In their inf-sup stability proof they require the mesh to be  quasi-uniform. In addition,  the maximum mesh size is assumed to be sufficiently small. In this paper we give an alternative proof of the inf-sup stability that relaxes these restrictions. To be more precise, we prove that the Scott-Vogelius finite element spaces for polynomial order $k \ge 4$ are inf-sup stable 
 assuming only that the family of meshes are 
non-degenerate (shape regular).  
One key aspect in the new proof is to use the stability of the $P^2-P^0$ (or the Bernardi-Raugel \cite{bernardi1985analysis}) finite element spaces. As a result the proof becomes significantly shorter.
In the last paragraph of \cite{falk2013stokes}, a modification
of their techniques is sketched that provides a proof of the inf-sup stability for
the Scott-Vogelius elements that is different from both ours and the original proof.

Recently there has been interest in developing finite element methods that produce divergence free velocities or have better mass conservation properties; see for example 
\cite{
 arnold1992quadratic,
 cockburn2005locally,
 evans2013isogeometric,
 falk2013stokes,
 guzman2012family,
 guzman2014conforming2,
 guzman2014conforming,
2016arXiv160508657H,
john2016divergence,
 linke2009collision,
 lovadina2015divergence,
 neilan2015discrete,
 ref:QinThesis,
 tai2006discrete,
 zhang2007family,
zhang2011divergence}.  
In particular, the review paper \cite{john2016divergence} discusses in depth the effects of mass conservation in simulations. There have also been 
  extensions of the Scott-Vogelius elements to three dimensions \cite{zhang2007family, zhang2011divergence, neilan2015discrete}, although a full general result is  still  out of reach. One difficulty lies in generalizing the concept of singular (or non-singular) vertices to three dimensions; see \cite{neilan2015discrete}.  


To better describe the key differences between the proof of inf-sup stability in this article compared to the original proof found in \cite{scott1985norm}, we  review
the proof of \cite{scott1985norm}.  Roughly speaking, in \cite{scott1985norm} given a pressure function $p$ from the discrete  space, one wants to find a velocity vector 
field  $v$ from the discrete velocity space such that $\dive v$ is close to $p$ 
and $\|v\|_{H^1(\Omega)} \le C \|p\|_{L^2(\Omega)}$. 
To do this, the proof in \cite{scott1985norm} follows roughly three steps. In the first step a velocity field $v_1$ is found  so that $p_1=\dive v_1 -p$ vanishes 
at all vertices.  
At this step it would be desirable to find a vector field $w$ so that $\dive w$ has the same average as $p_1$ on each triangle and such that $\dive w$ vanishes at the vertices. This can be done, however, it is not clear how to do this in a stable way. Therefore,   alternatively in the second step,   a continuous piecewise linear pressure  function $\tilde{p}$ is introduced so that   $p_2=p_1-\tilde{p}$ has average zero on non-overlapping patches of roughly size $Kh$ where $K$ is a sufficiently large constant. Then one finds a vector field $v_2$ so that $\dive v_2=p_2$. As a result

\begin{equation*}
\dive (v_1+v_2)=p-\tilde{p} \text{ on } \Omega.
\end{equation*}
The final step is to find a vector field $v_3$ so that $\dive v_3 \approx \tilde{p}$, where here quasi-uniformity and sufficiently small mesh size is used. Then one arrives at
\begin{equation*}
\dive (v_1+v_2+v_3) \approx p \text{ on } \Omega.
\end{equation*}

In contrast, we reverse the order of the steps. 
First we find a piecewise quadratic $v_1$ such that $p_1=\dive v_1-p$ 
 has average zero on each triangle . This can be done in a stable way, 
using the Bernardi-Raguel finite elements \cite{bernardi1985analysis} (or the $P^2-P^0$ finite element space).  
Then one  finds $v_2$ that is piecewise quartic so that $p_2=\dive v_2-p_1$ vanishes at the vertices and has average zero at each triangle, in which case we require that $\dive v_2$ has average zero on each triangle.  To construct $v_2$   we combine two basis functions that are used implicitly in \cite{scott1985norm}. Finally, following \cite{vogelius1983right} a local argument will find $v_3$ so that $\dive v_3=p_2$.  Hence, $\dive(v_1+v_2+v_3)=p$.

The only reason that we are restricted to $k \ge 4$ is that the $v_2$ constructed above is piecewise quartic. 
In a subsequent paper we show that if we impose further mild restrictions on non-singular vertices then we can choose $v_2$ to be piecewise cubic and, hence, prove inf-sup stability in the case $k=3$.

The paper is organized as follows: In the following section we begin with Preliminaries. 
In Section \ref{proof} we prove the inf-sup stability for $k \ge 4$. 

\section{Preliminaries}\label{preliminaries}

We assume $\Omega$ is a polygonal domain in  two dimensions. We let $\{ \mathcal{T}_h\}_h$ be a family of nondegenerate (shape regular) triangulations of $\Omega$; see \cite{brenner2007mathematical} .   The set of vertices and the set of internal edges are denoted by
\begin{alignat*}{1}
\Sh=&\{ x:  x  \text{ is a vertex of } \Th \}, \\
\Eh=&\{ e: e \text{ is an edge of }  \Th \text{ and } e \not\subset \partial \Omega \}.\\ 
 \end{alignat*}

For every $z \in \Sh$ the function $\psi_z$  is the continuous, piecewise linear 
Lagrange basis function corresponding to the vertex $z$. That is, for every $y \in \Sh$
\begin{equation}\label{eqn:weneedthis}
 \psi_z(y)= 
\begin{cases}
1 \quad \text{ if } y=z,\\
0 \quad \text{ if } y \not=z. 
\end{cases}
\end{equation}
We also define the internal edges and triangles that have $z \in \Sh$ as a vertex:
\begin{alignat*}{1}
\Eh(z)=&\{ e \in \Eh: z \text{ is a vertex of } e \}, \\
\Th(z)=&\{ T \in \Th: z \text{ is a vertex of } T \}.  
\end{alignat*}
Finally, we define the patch
\begin{equation*}
\Omega_h(z)= \bigcup_{T \in \Th(z)} T.
\end{equation*}
The diameter of this patch is denoted by $h_z= \text{diam}(\Omega_h(z))$.

In order to define the pressure space we need to define singular and non-singular vertices. 
Let $z \in \Sh$  and suppose that $\Th(z)=\{ T_1, T_2, \ldots T_N \}$. If $z$ is a boundary vertex then we enumerate the triangles such that $T_1$ and  $T_N$ have a boundary edge.  Moreover, we enumerate them so that $T_j, T_{j+1}$ share an edge for $j=1, \ldots N-1$ and  $T_N$ and $T_1$ share an edge in the case $z$ is an interior vertex. Let $\theta_j $ denote the angle between the edges of $T_j$ originating from $z$.   We define 
\[
\Gamma(z)= 
\begin{cases}
\max \{ |\theta_1+\theta_{2}- \pi|,  \ldots, |\theta_{N-1}+\theta_{N}- \pi|, |\theta_N+\theta_1-\pi|\} & \text{ if }  z \text{ is an interior vertex } \\
\max \{ |\theta_1+\theta_{2}- \pi|,   \ldots, |\theta_{N-1}+\theta_{N}- \pi| \} & \text{ if }  z \text{ is a boundary vertex }.
\end{cases}
\]

Later, we will also need the following definition:
\[
\Theta(z)= 
\begin{cases}
\max \{ |\sin(\theta_1+\theta_{2})|,  \ldots, |\sin(\theta_{N-1}+\theta_{N})|, |\sin(\theta_N+\theta_1)| \}& \text{ if }  z \text{ is an interior vertex } \\
\max \{ |\sin(\theta_1+\theta_{2})|,  \ldots, |\sin(\theta_{N-1}+\theta_{N})| \} & \text{ if }  z \text{ is a boundary vertex }.
\end{cases}
\]

\begin{Def}
A vertex $z \in \Sh$ is a singular vertex if $\Gamma(z)=0$. It is non-singular if $\Gamma(z) >0$.
\end{Def}

\begin{figure}
\centerline{\qquad\qquad\includegraphics[scale=.8]{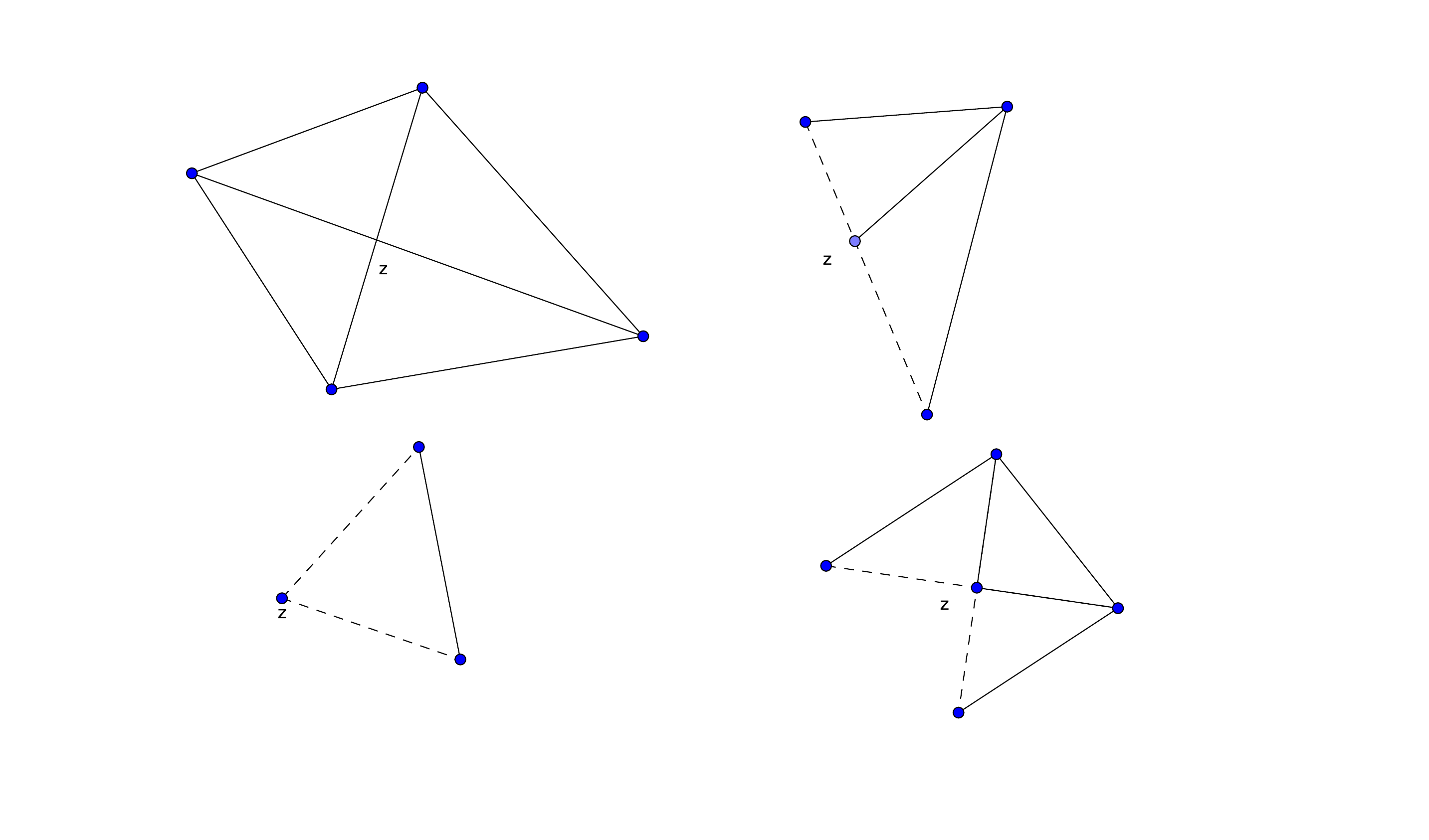}}
\footnotesize
\ra{1.1}
\vspace{-40pt}
\caption{Example of singular vertices $z$. Dashed edges denote boundary edges.}
\label{triangle4}
\end{figure}

We denote all the non-singular vertices by
\begin{equation*}
\Sh^1=\{ x \in \Sh: x \text{ is non-singular } \},
\end{equation*}
and all singular vertices by $\Sh^2=\Sh \backslash \Sh^1$.

Let $q$ be a function such that $q|_T \in C(\overline{T})$\footnote{We define
$ C^k(\overline{T})$ to be the subset of $C^k(T)$ consisting of functions 
with continuous limits on $\overline{T}$, $k\ge 0$.} for all $T \in \Th$. 
For each vertex $z \in \Sh^2$ define
\begin{equation*}
A_h^z(q)= \sum_{j=1}^N (-1)^{N-j} q|_{T_j}(z).
\end{equation*}

Now we are ready to define the Scott-Vogelius finite element spaces for $k \ge 1$:
\begin{alignat*}{1}
V_h^k=&\{v \in [C_0(\Omega)]^2: v|_T \in [P^k(T)]^2 \Forall T \in \Th\}, \\
Q_h^{k-1}=& \{ q \in L_0^2(\Omega): q|_T \in P^{k-1}(T) \Forall T \in \Th,\; A_h^z(q)=0 
\Forall  z  \in \Sh^2\}. 
\end{alignat*}
Here $P^k(T)$ is the space of polynomials of degree less than or equal to $k$ defined on $T$. Also, $L_0^2(\Omega)$ denotes the subspace of $L^2$ of functions that have average zero on $\Omega$.

Note that if all the vertices are non-singular, then $Q_h^{k-1}$ is the space of 
discontinuous piecewise polynomials of degree $k-1$ with average zero on $\Omega$. If singular vertices exist  then the pressure space is constrained on those vertices.  
Finally, the only way an interior vertex $z$ is singular is if $z$ has four edges coming out of it and they lie on two lines.

The definition of $Q_h^{k-1}$ is natural as the following result which was proved in \cite{scott1984conforming} shows.
\begin{lemma}\label{lemma1}
For $k \ge 1$, it holds
\begin{equation*}
\dive V_h^k \subset Q_h^{k-1}.
\end{equation*}
\end{lemma}

This lemma is a consequence of the following result.  
\begin{lemma}\label{lemmasingular}
Let $v$ be a vector field such that $v|_{\overline{T}} \in C^1(\overline{T})$ for every $T \in \mathcal{T}_h$. In addition, assume $v \in C(\Omega)$ and $v$ vanishes on $\partial \Omega$  . Then, $A_h^z(\dive v)=0 $ for every singular vertex $z \in S_h^2$.  
\end{lemma}
We prove this result for completeness in the appendix
following the argument given in \cite{scott1984conforming}.

The goal of this article is to prove the inf-sup stability of the pair $V_h^k, Q_h^{k-1}$ for $ k \ge 4$ . We recall the definition of inf-sup stability.
\begin{Def}
The pair of spaces $V_h^k, Q_h^{k-1}$ are inf-sup stable on a family of triangulations 
$\{\Th\}_h$ if there exists $\beta>0 $ such that for all $h$
\begin{equation}\label{inf-sup}
\beta \, \|q\|_{L^2(\Omega)} \le \sup_{ v \in V_h^{k},  v \not\equiv 0}  \frac{\int_{\Omega} q \dive v  \, dx}{ \|v\|_{H^1(\Omega)}}  \quad \quad \forall q \in  Q_h^{k-1}. 
\end{equation}
\end{Def}

Let $z \in \Sh$ and $e \in \Eh(z)$  where $e=\{ z, y\}$.
Define $t_e^z=\frac{(y-z)}{|e|}$ to be the unit vector that is tangent to $e$.   
It is clear that  
\begin{equation}\label{tpsi}
t_e^z \cdot \nabla \psi_y= \frac{-1}{|e|}  \text{ on }  e. 
\end{equation}

We will also need to compute the derivatives of $\psi_y$  in all directions. Suppose that $T \in \Th(y)$ and let $g$ be the edge of $T$ that is opposite to $y$. If we let $n_T^y$ 
be the unit normal vector to $g$ that points out of $T$ then
\begin{equation}\label{nablapsi}
\nabla \psi_y|_T = - \frac{1}{h_T^y} n_T^y,   
\end{equation}
where $h_T^y$  is the distance of $y$ to the line defined by the edge $g$. If 
$z$ is another vertex of $T$ and  we denote the edge $e=\{ z, y\}$  then  a simple calculation gives
\begin{equation}\label{hTj}
h_T^y= \sin(\theta) |e|
\end{equation}
where $\theta$ is the angle between the edges of $T$ originating from $z$. See Figure \ref{triangle2} for an illustration.

\begin{figure}
\vspace{-30pt}
\centerline{\includegraphics[scale=1.2]{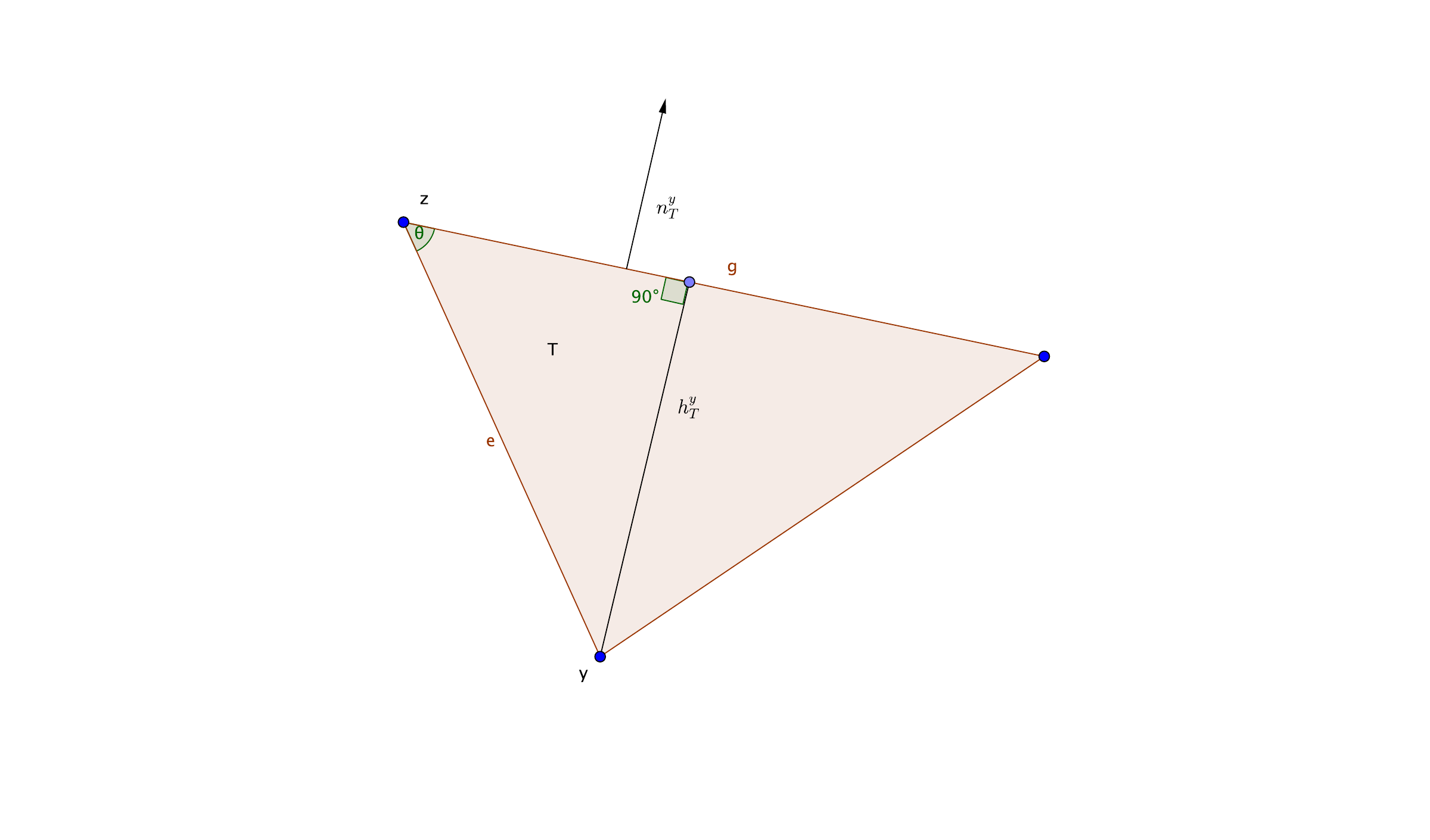}}
\footnotesize
\ra{1.1}
\vspace{-90pt}
\caption{Illustration of geometric quantities in one triangle}
\label{triangle2}
\end{figure}

\section{Establishing the inf-sup condition}\label{proof}

\subsection{Preliminary stability results}

In order to prove the inf-sup stability of the Scott-Vogelius finite element spaces  we need two well known results. The first follows from the stability of the Bernardi-Raugel \cite{bernardi1985analysis} finite element space.
\begin{proposition}\label{prop1}
Let $k \ge 1$. There exists a constant $\alpha_1$ such that for every $p \in Q_h^{k-1}$  there exists a $v \in V_h^2$ such that 
\begin{equation*}
\int_{T} \dive v \, dx =\int_T p \, dx \quad \text{ for all } T \in \Th,
\end{equation*}  
and 
\begin{equation*}
\|v\|_{H^1(\Omega)} \le \alpha_1  \|p\|_{L^2(\Omega)}. 
\end{equation*}
The constant $\alpha_1$ is independent of $p$ and only depends on the shape regularity of the mesh and $\Omega$.
\end{proposition}

The second result we need is contained in Lemma 2.5 of \cite{vogelius1983right}. A three dimensional analog is given in \cite{neilan2015discrete}. We provide the proof here for completeness.
\begin{proposition}
Let $k \ge 1$. There exists a constant $\alpha_2$ such that for every $T \in \Th$ and any $p_T \in P^{k-1}(T)$ with $p$ vanishing on the vertices of $T$ and $\int_T p_T dx=0$ there exists a $v_T \in [P^k(T)]^2$ with $v_T =0$ on $\partial T$ such that  
\begin{equation*}
\dive v_T =p_T \text{ on } T,
\end{equation*}
and
\begin{equation*}
\|v_T\|_{H^1(T)} \le \alpha_2 \|p_T\|_{L^2(T)}.
\end{equation*}
Here the constant $\alpha_2$  depends only
on the shape regularity of the mesh and $k$. 
\end{proposition}

\begin{proof}
Consider the spaces
\begin{alignat*}{1}
M^{k-1}(T)= & \{q\in P^{k-1}(T):  q \text{ vanishes at the vertices of } T \text{ and }  \int_T q(x)\,dx=0 \} \\
\hbox{and}\quad B^k(T)=&\{v \in [P^k(T)]^2: v=0 \text{ on }  \partial T\} =\{b_T v: v \in  [P^{k-3}(T)]^2\}.
\end{alignat*}
Here $b_T$ is the cubic bubble of $T$ that vanishes on $\partial T$. 
It is easy to show that the 
\[
\text{ dim } M^{k-1}(T) {=}
\begin{cases}
\frac{1}{2}k(k+1) - 4 &  \text{ for } k \ge 3, \\
 0  & \text{ for } k \le 2,
\end{cases}
\]
and
 \[
\text{ dim } B^{k}(T) {=}
\begin{cases}
(k-2)(k-1)&  \text{ for } k \ge 3, \\
 0  & \text{ for } k \le 2.
\end{cases}
\]
Moreover,  it is clear that $\dive B^k(T)\subset M^{k-1}(T)$. Finally, we set 
$Z^{k}(T)=\{ v \in B^k(T): \dive v \equiv 0 \text{ on }  T \} $
then a simple argument gives
\begin{equation*}
Z^{k}(T)= \{\text{ curl } (\psi b_T^2) : \psi  \in P^{k-5}(T) \}.
\end{equation*}
Hence,
 \[
\text{ dim } Z^{k}(T) {=}
\begin{cases}
\frac{1}{2}(k-4)(k-3) &  \text{ for } k \ge 5, \\
 0  & \text{ for } k \le 4.
\end{cases}
\]
We claim that $  M^{k-1}(T) = \dive B^k(T)$. To show this, we will use a dimension count. First note that
\begin{equation*}
\text{ dim } [B^k(T)] = \text{ dim } [\dive B^k(T)] + \text{ dim } [Z^{k}(T)].
\end{equation*}

For $k \le 4$ we know that $ Z^{k}(T)$ is empty so that $\text{ dim } [\dive B^k(T)]= \text{ dim } [B^k(T)] = \text{ dim } [M^{k-1}(T)]$. For $k \ge 4$
\begin{equation*} 
\begin{split}
 \text{ dim } [\dive B^k(T)] &= \text{ dim } [B^k(T)]- \text{ dim } [Z^{k}(T)]\\
&=(k-2)(k-1)-\frac{1}{2}(k-4)(k-3)\\
&=\frac{1}{2} k^2 -3k+2 +(7/2)k-6 \\
&=\frac{1}{2} k^2 +\frac{1}{2} k -4=\dim [M^{k-1}(T)].
\end{split}
\end{equation*}

Thus we have proved that for any $p_T \in M^{k-1}(T)$ there exists a $v_T$ such that  
\begin{equation*}
\dive v_T =p_T \text{ on } T.
\end{equation*}
The bound 
\begin{equation*}
\|v_T\|_{H^1(T)} \le \alpha_2 \|p_T\|_{L^2(T)}
\end{equation*}
 follows from a scaling argument mapping to the reference element  with  the Piola transform which preserves the divergence. 
\end{proof}

Summing up the result for every $T \in \mathcal{T}_h$ we can prove the following lemma.
\begin{lemma}\label{lemma3}
Let $k \ge 1$. Let $\alpha_2$ be the constant from the previous proposition. For every $p \in Q_h^{k-1}$ such that $p(z) =0$ for all $z \in \Sh$ and $\int_T p \, dx =0$ for all $T \in \Th$ there exists $v \in V_h^k$ such that 
\begin{equation*}
\dive v = p\text{ on }  \Omega,
\end{equation*}
and
\begin{equation*}
\|v\|_{H^1(\Omega)} \le \alpha_2 \|p\|_{L^2(\Omega)}.
\end{equation*}
\end{lemma}

\subsection{ Interpolating vertex values: Fundamental Vector Fields}
We will need to define vector fields that will help in interpolating pressure 
vertex values. To do this, we first need to define the following functions.

For every $z \in \Sh$  and $e \in \Eh(z)$ with $e=\{ z, y\}$ we define the two functions
\begin{alignat*}{1}
\eta_e^z= & \psi_z^2 \psi_y  \\
\gamma_e^z=&\eta_e^z- \frac{5}{2} \psi_z^2 \psi_y^2,
\end{alignat*}
where $\psi_z$ is defined in \eqref{eqn:weneedthis}.
Let $T_1$ and $T_2$ be the two triangles that have $e$ as an edge.
Then we can easily verify the following:
\begin{subequations}\label{eta}
\begin{alignat}{1}
& \text{ support } \eta_e^z  \subset T_1 \cup T_2 ,  \text{ support }  \gamma_e^z \subset T_1 \cup T_2,  \label{eta1} \\
& \nabla(\eta_e^z) (\sigma )=0= \nabla(\gamma_e^z) (\sigma ) \quad \text{ for  } \sigma \in \Sh  \text{ and }   \sigma \neq z, \label{eta2} \\
 &\int_e \gamma_e^z \, ds =0 \label{eta3}.
\end{alignat}
\end{subequations}

For example, to prove the last equation we used that
\begin{alignat}{1}
\int_e \psi_z^2 \psi_y \, ds=& \frac{|e|}{12} , \label{psi1} \\
\int_e \psi_z^2 \psi_y^2 \, ds=& \frac{|e|}{30}, \label{psi2}
\end{alignat}
which can be done by transforming the edge $e$ to the unit interval. 
 
The first vector field we define is
\begin{equation}\label{wei}
w_e^z=-|e| t_e^z \eta_e^z.
\end{equation}

We are going to need to consider $w_e^z$ and vector fields of the form $c \gamma_e^z$. 
The following lemmas collect properties of these functions. 

\begin{lemma}
Let $z \in \Sh$  and $e \in \Eh(z)$ with $e=\{ z, y\}$ and denote the two triangles that have $e$ as an  edge as $T_1$ and $T_2$. Let $v= c |e| \gamma_e^z$ where $c$ is a constant vector and let $w_e^z$ be given by \eqref{wei}. It  holds,

\begin{subequations}\label{w}
\begin{alignat}{1}
& w_e^z \in V_h^3,  \quad v \in V_h^4 \label{w1} \\
& \text{ support } w_e^z  \subset T_1 \cup T_2 ,  \text{ support } v  \subset T_1 \cup T_2 \label{w2} \\
& \int_K \dive w_e^z\, dx= 0 =\int_K \dive v \, dx  \quad \text{ for all } K  \in \Th, \label{w3}  \\
& \dive v(\sigma )=0=\dive w_e^z(\sigma ) \quad \text{ for } \sigma \in \Sh \text{ and } \sigma \neq z, \label{w4}\\
& \dive v|_{T_s} (z)=|e|\, c \cdot \nabla \psi_y|_{T_s}=  \frac{-|e|}{h_{T_s}^y}  c \cdot n_{T_s}^y  \quad \text{ for } s=1,2 \label{w5} \\  
& \dive w_{e}^z|_{T_s} (z)= 1 \quad \text{ for } s=1,2 \label{w6}\\
& \|\nabla w_e^z\|_{L^2(T_1 \cup T_2 )} \le  C h_z, \quad\hbox{and}\quad  \|\nabla v\|_{L^2(T_1 \cup T_2)}  \le C\, h_z  \, |c|. \label{w7} 
\end{alignat}
\end{subequations}
The constant $C$ only depends on the shape regularity.
\end{lemma}

\begin{proof}
The results \eqref{w1} and \eqref{w2} follow directly from the definitions of $\eta_e^z$ and $\gamma_e^z$.  To prove \eqref{w3} note that it trivially holds for $K \in \Th$ if $K$ is not $T_1$ or $T_2$. Therefore, let $K=T_s$ (for $s=1, 2$) then, using  \eqref{w2}  and integration by parts we get
\begin{equation*}
 \int_K \dive v \, dx = \int_{e}  v \cdot n \, ds
\end{equation*}
where $n$ is the unit vector normal to $e$ pointing out of $K$.
Therefore,  by \eqref{eta3} 
\begin{equation*}
 \int_K \dive v \, dx  = |e| \cdot (c \cdot n) \int_{e} \gamma_e^z \, ds=0.
\end{equation*}
Since $t_e^z \cdot n =0$ we can also show
\begin{equation*}
 \int_K \dive w_e^z\, dx= 0.
\end{equation*}
The equations \eqref{w4} follow from \eqref{eta2}. 

To prove \eqref{w5}, we use that $\psi_y(z)=0$  and $\psi_z(z)=1$ to get 
\begin{equation*}
\begin{split}
\dive v|_{T_s}(z) &= |e| \psi_z^2(z) \big(1- 5 \psi_y(z)\big) c \cdot\nabla\psi_y|_{T_s} 
+|e|2 \psi_z(z)\big(\psi_y(z)-\frac{5}{2} \psi_y(z)^2\big) c \cdot\nabla\psi_z|_{T_s} \\
&=|e| \, c \cdot \nabla  \psi_y|_{T_s}.
\end{split}
\end{equation*}
The result follows from \eqref{nablapsi}.
Similarly, 
\begin{equation*}
 \dive w_e^z |_{T_s} (z)= -|e| t_e^z \cdot \nabla \psi_y|_{T_s}  =-|e| t_e^z \cdot  \nabla \psi_y|_{e} = 1,  
\end{equation*}
where we used that $\psi_y$ is continuous along $e$ and therefore $ t_e^z \cdot \nabla \psi_y|_{T_1}= t_e^z \nabla  \cdot \psi_y|_{T_2}$. Finally, we used \eqref{tpsi}. 

To prove \eqref{w7}, we use the Cauchy-Schwarz inequality, and an inverse estimate
\begin{equation*}
\|\nabla v\|_{L^2(T_1 \cup T_2)} \le C \, h_z \|\nabla v\|_{L^\infty(T_1 \cup T_2)}  \le C \, \|v\|_{L^\infty(T_1 \cup T_2)}  \le C \, |c| \, |e|\,  \, \|\gamma_e^z\|_{L^\infty(T_1 \cup T_2)}  \le  C \, h_z\, |c|.
\end{equation*}
Here we used the shape regularity of the mesh.  The bound for $w_e^z$ is similar.
\end{proof}

We define the other fundamental vector field in the following lemma.
\begin{lemma}\label{lemmavz}
For every $z \in \Sh^1$ and $T \in \Th(z)$ there exists a $v_T^z \in V_h^4$ with the following properties. 
\begin{subequations}\label{lemma1vertex}
\begin{alignat}{1}
&\dive v_T^z (\sigma)= 0 \quad \text{ for all } \sigma \neq z,  \label{lemma1vertex1} \\
& \dive v_T^z|_T(z)=1, \quad \text{ and }  \dive v_T^z|_K (z)=0  \quad \text{ for all } K \in \Th(z) \text{ and } K \neq T \label{lemma1vertex2} \\
& \text{ support } v_T^z \subset  \Omega_h(z),  \label{lemma1vertex3} \\
& \int_K \dive v_T^z \, dx=0 \quad \text{ for all } K \in \Th. \label{lemma1vertex4}
\end{alignat}
\end{subequations}

The following bound holds
\begin{equation}\label{lemma1vertex5}
 \|\nabla v_T^z \|_{L^2(\Omega_h(z))} \le C \, h_z\Big(\frac{1}{\Theta(z)}+1\Big).
\end{equation}
The constant $C$ is independent of  $h$ and only depends on the shape regularity and $k$.
 
\end{lemma} 
\begin{proof}
We adopt the notation  given in the preliminary section which we recall here. 
For $z \in \Sh$ we enumerate the triangles that have $z$ as a vertex: $\Th(z)=\{ T_1, T_2, \ldots T_{N} \}$. 
If $z$ is a boundary vertex then we enumerate the triangles such that $T_1$ and  $T_N$ each have a boundary edge.  Moreover, we enumerate them so that $T_j, T_{j+1}$ share an edge  $e_j$ ,  for $j=1, \ldots N-1$ and  $T_{N}$ and $T_1$  share an edge $ e_{N}$ in the case $z$ is an interior vertex. Let $\theta_j $ denote the angle between the edges of $T_j$ originating from $z$.    We let $1 \le s \le N-1$ be such that $ |\sin(\theta_s+ \theta_{s+1})|=\Theta(z)$. 
Without loss of generality we can assume that $s \neq N$; this is immediate
for a boundary vertex, and for an interior vertex
 we can enumerate the triangles accordingly.
Let  $e_s=\{ z, y\}$, then recall that $n_{T_s}^y$ and $n_{T_{s+1}}^y$ are the unit normal vectors pointing out of $T_{s}$, $T_{s+1}$, respectively,  at the edges opposite to $y$. 
Let $t_{s+1}$ be the tangent vector to $e_{s+1}$ pointing 
away from  $z$ orthogonal $n_{T_{s+1}}^y$ (i.e.\, $ t_{s+1} \cdot n_{T_{s+1}}^y=0$). See Figure \ref{triangle1} for an illustration.

\begin{figure}
\vspace{-100pt}
\centerline{\qquad\qquad\includegraphics[scale=1.0]{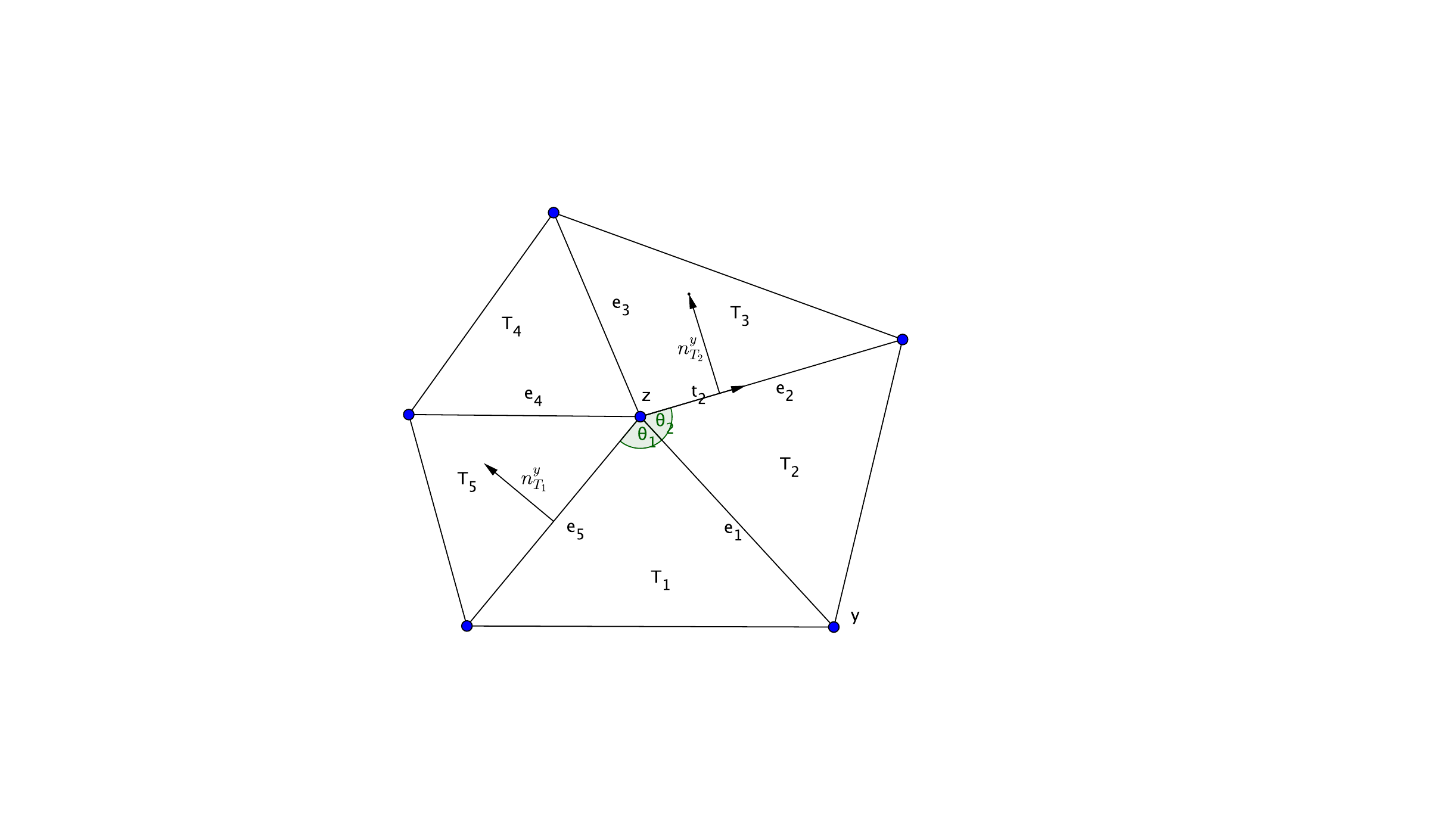}}
\footnotesize
\ra{1.1}
\vspace{-100pt}
\caption{Example, on non-singular vertex $z$, $N=5$, $s=1$.}
\label{triangle1}
\end{figure}

 We need to define $v_{T_j}^z$ for $1 \le j \le N$. We start by defining

\begin{equation*}
v_{T_{s}}^z= \frac{ |e| \, \sin(\theta_s)}{ \Theta(z)}   t_{s+1} \eta_{e_s}^z .
\end{equation*}

Then, for every $1 \le \ell \le N-s $ we define
\begin{equation*}
v_{T_{s+\ell}}^z= (-1)^ {\ell-1} \big(v_{T_{s}}^z- w_{e_{s}}^z+w_{e_{s+1}}^z+  \cdots+ (-1)^{\ell-1} w_{e_{s+ \ell-1}}^z\big).    
\end{equation*}

Also, for   $1 \le  \ell \le s-1$
\begin{equation*}
v_{T_{s-\ell}}^z= (-1)^ {\ell-1} \big(v_{T_{s}}^z- w_{e_{s-1}}^z+w_{e_{s-2}}^z+  \cdots+ (-1)^{\ell-1} w_{e_{s-(\ell -1)}}^z\big).    
\end{equation*}

With these definitions, \eqref{lemma1vertex1}, \eqref{lemma1vertex3} and \eqref{lemma1vertex4} clearly follow from \eqref{w4},\eqref{w2} and \eqref{w3},  respectively.  We are left to prove \eqref{lemma1vertex2} and \eqref{lemma1vertex5}.  Let us first prove this for $T=T_{s}$. By the definition of $\eta_{e_s}^z$  it has support in $T_{s} \cup T_{s+1}$ and, therefore, we only need to consider $K=T_{s}$ and $K={T_{s+1}}$. First, by \eqref{w5}
\begin{equation}\label{eqsp1_1}
\dive v_{T_{s}}^z|_{T_{s+1}}(z)=-\frac{ \sin(\theta_{s})}{ \Theta(z)}   \frac{|e|}{h_{T_{s+1}}^y}  t_{s+1}  \cdot n_{T_{s+1}}^{y}=0.
\end{equation}
On the other hand,
\begin{equation}\label{eqsp1_2}
\dive v_{T_{s}}^z|_{T_{s}}(z)= -\frac{\, \sin(\theta_{s})}{ \Theta(z)}   \frac{|e|}{h_{T_{s}}^y}  t_{s+1}  \cdot n_{T_{s}}^{y} =- \frac{ t_{s+1}  \cdot n_{T_{s}}^{y}}{\Theta(z)}=1,
\end{equation}
where we used that $ t_{s+1}  \cdot n_{T_{s}}^{y}= \cos(\theta_s+\theta_{s+1}+\frac{\pi}{2})=-\sin(\theta_s+ \theta_{s+1})=-\Theta(z)$. The inequality  \eqref{lemma1vertex5} follows from \eqref{w7} to get 
\begin{equation}\label{boundsp1}
\|\nabla v_{T_{s}}^z\|_{L^2(\Omega_h(z))} \le   \frac{ C |e_{s}|} {\Theta(z)} \le \frac{ C h_z}{\Theta(z)}.
\end{equation}

For $T=T_{s+\ell}$, $1 \le \ell \le N-s $ and $T=T_{s-\ell}$ for $1 \le  \ell \le s-1$  we can use \eqref{eqsp1_1}, \eqref{eqsp1_2} and \eqref{w6} to prove \eqref{lemma1vertex2}. The bound \eqref{lemma1vertex5} follows from \eqref{boundsp1} and \eqref{w7}, using $N$ is bounded depending only on the shape regularity.  
\end{proof}

Using these vector fields we can prove a crucial lemma.  

\begin{lemma}\label{lemma0}
For every $p \in Q_h^{k-1}$ and $z \in \Sh$ there exists  a $v$ such that
the following properties hold:
\begin{subequations}\label{lemmavertex}
\begin{alignat}{1}
&\dive v (\sigma)= 0 \quad \text{ for all } \sigma \neq z,  \label{lemmavertex1} \\
& \dive v|_T (z)= p|_T(z) \quad \text{ for all } T \in \Th(z), \label{lemmavertex2} \\
& \text{ support } v \subset  \Omega_h(z),  \label{lemmavertex3} \\
& \int_K \dive v \, dx=0 \quad \text{ for all } K \in \Th. \label{lemmavertex4}
\end{alignat}
\end{subequations}
If $z$ is a singular vertex then $v \in V_h^3$ 
with the bound
\begin{equation}\label{lemmavertex5}
 \|\nabla v \|_{L^2(\Omega_h(z))} \le C \|p\|_{L^2(\Omega_h(z))}.
\end{equation}
If $z$ is a non-singular vertex then $v \in V_h^4$ 
with the bound
\begin{equation}\label{lemmavertex6}
 \|\nabla v \|_{L^2(\Omega_h(z))} \le \frac{C}{\Theta(z)+1} \|p\|_{L^2(\Omega_h(z))}.
\end{equation}
The constant $C$ is independent of $p$ and  $h$ and only depends on the shape regularity and $k$.

\end{lemma}

\begin{proof}
We adopt the notation from  the proof of the previous lemma.  We set $a_j=p|_{T_j}(z)$. First, suppose that $z$ is a singular vertex.  Then, we know by the definition of $Q_h^{k-1}$
\begin{equation}\label{altsum}
\sum_{j=1}^{N} (-1)^{N-j} a_j=0.
\end{equation}

We define 
\begin{equation*}
v=\sum_{j=1}^{N-1}  b_j w_{e_j}^z
\end{equation*}
where  we set
\begin{equation*}
b_{j}=\bigg(\sum_{\ell=1}^{j} (-1)^{j-\ell} a_\ell\bigg) \text{ for all } 1\le j\le N-1.
\end{equation*}
We immediately see that \eqref{lemmavertex1} , \eqref{lemmavertex3} and \eqref{lemmavertex4} from \eqref{w4}, \eqref{w2}, \eqref{w3}, respectively. Moreover, \eqref{w1} gives that $v \in V_h^3$.

Using \eqref{w6} we have 
\begin{equation*}
\dive v|_{T_j}(z)= b_{j-1}+  b_{j}= a_j \quad \text{ for } 2 \le j \le N-1. 
\end{equation*}
For $j =1$ we have 
\begin{equation*}
\dive v|_{T_1}(z)=  b_{1}= a_1. 
\end{equation*}
Also, 
\begin{equation*}
\dive v|_{T_{N}}(z)=  b_{N-1} =  \sum_{\ell=1}^{N-1} (-1)^{N-1-\ell} a_\ell= a_{N},
\end{equation*}
where we used \eqref{altsum}.  Therefore, we have shown \eqref{lemmavertex2}.
We see from \eqref{w7} that 
\begin{equation*}
\|\nabla v\|_{L^2(\Omega_h(z))} \le C \, h_z\bigg(\sum_{j=1}^{N-1} b_j^2\bigg)^{1/2}
\le C \, h_z \sum_{j=1}^{N-1} |b_j| \le C h_z N \sum_{j=1}^{N} |a_j| 
\le C h_z N \sum_{j=1}^{N} \|p\|_{L^\infty(T_j)},
\end{equation*}
where the equivalence of norms uses
that $N$ is bounded depending only on the shape regularity.  
This bound on $N$ also implies
inequality \eqref{lemmavertex5}  after applying the inverse estimate
\begin{equation*}
\|p\|_{L^\infty(T_j)} \le \frac{C}{h_{T_j}} \|p\|_{L^2(T_j)}.
\end{equation*}

Next, we assume $z$ is a non-singular singular vertex.  In this case, we define
\begin{equation*}
v=\sum_{j=1}^{N} a_j v_{T_j}^z.  
\end{equation*}

Clearly, \eqref{lemmavertex1},  \eqref{lemmavertex2},  \eqref{lemmavertex3},  \eqref{lemmavertex4} follow from \eqref{lemma1vertex1},  \eqref{lemma1vertex2},  \eqref{lemma1vertex3},  \eqref{lemma1vertex4}. Using  \eqref{lemma1vertex5} we get
\begin{equation*}
\|\nabla v\|_{L^2(\Omega_h(z))} \le  C \, h_z \Big(\frac{1}{\Theta(z)}+1\Big) \|p\|_{L^\infty(\Omega_h(z))}. 
\end{equation*}
 The inequality \eqref{lemmavertex6} follows after applying inverse estimates. 

\end{proof}

We can use the previous lemma to prove a global result. First, we  define 
 \begin{equation*}
 \tmin= \min_{z \in \Sh^1} \Theta(z).
 \end{equation*}

\begin{lemma} \label{lemma2}
For every $p \in Q_h^{k-1}$ there exists $v \in V_h^4$
\begin{subequations}\label{globalvertex}
\begin{alignat}{1}
& (\dive v-p)(z)=0 \quad \text{ for all } z \in \Sh, \label{globalvertex1} \\
& \int_K \dive v \, dx=0 \quad \text{ for all } K \in \Th \label{globalvertex2} 
\end{alignat}
\end{subequations}
and
\begin{equation}\label{globalvertex3}
 \|\nabla v \|_{L^2(\Omega)} \le C_1\Big(\frac{1}{\tmin}+1\Big) \|p\|_{L^2(\Omega)},
\end{equation}
where the constant $C_1$ is independent of $p$ and depends only on the shape regularity of the meshes and $k$. 
\end{lemma}

\begin{proof}
Let $p \in Q_h^{k-1}$ be given. Given $z \in \Sh$ let $v_z$ denote  the vector field satisfying the properties of the previous lemma.  
Then, we set $v= \sum_{z \in \Sh} v_z$.  Clearly, from the previous lemma, \eqref{globalvertex1} holds. 
Finally, since only three $v_z$'s are non-zero on each given triangle $T$ we can easily show that 
\begin{equation*}
\begin{split}
\|\nabla v\|_{L^2(\Omega)}^2 &= \sum_{T\in\Th}\|\nabla v\|_{L^2(T)}^2 
 = \sum_{T\in\Th}\sum_{z \in T}\|\nabla v_z\|_{L^2(T)}^2 
 = \sum_{z \in \Sh} \| \nabla v_z \|_{L^2(\Omega_h(z))}^2 \\
&\le \frac{C}{\Theta_{\min}^{\,2}}  \sum_{  z \in \Sh^1} \|p\|_{L^2(\Omega_h(z))}^2+ C   \sum_{  z \in \Sh \backslash \Sh^1} \|p\|_{L^2(\Omega_h(z))}^2.
\end{split}
\end{equation*}

The inequality \eqref{globalvertex3} now easily follows.

\end{proof}

\subsection{ The Final Step}
We can now combine all the above results to prove the inf-sup condition. 
\begin{theorem}\label{mainthm4}
Suppose that our family of meshes $\{\Th\}_h$  is non-degenerate (shape regular). 
Then,  $Q_h^{k-1}, V_h^k $ satisfy the inf-sup condition \eqref{inf-sup} for $k \ge 4$ where the constant $\beta$ depends on $\tmin$ and $k$, but is independent of $h$. 
\end{theorem}

\begin{proof}
Let $p \in Q_h^{k-1}$. Let $v_1 \in V_h^2$ be from Proposition \ref{prop1} and let $p_1=p-\dive v_1$.   We have that $\int_T p_1 \, dx =0 $ for all $T \in \Th$. By Lemma \ref{lemma1} we have that $p_1 \in Q_h^{k-1}$.  Given $p_1$  let $v_2 \in V_h^4$   be the corresponding vector field from Lemma \ref{lemma2}. Then,  $p_2= p_1-\dive v_2$ satisfies $\int_{T} p_2 \, dx =0 $ for all  $T \in \Th$ and $p_2$ vanishes at all the vertices.  We can, therefore, apply Lemma \ref{lemma3} and have a $v_3 \in V_h^k$ so that $\dive v_3= p_2$ on $\Omega$. Setting $v=v_1+v_2 +v_3 \in V_h^k$ we have
\begin{equation*}
\dive v= p  \text{ on } \Omega,
\end{equation*}
and
\begin{equation*}
\|\nabla v\|_{L^2(\Omega)} \le  (\alpha_2 \|p_2\|_{L^2(\Omega)} + C_1\Big( \frac{1}{\tmin}+1\Big)\|p_1\|_{L^2(\Omega)} + \alpha_1 \|p\|_{L^2(\Omega)}) \le C_2\Big( \frac{1}{\tmin}+1\Big) \|p\|_{L^2(\Omega)},
\end{equation*}
where $C_2$  depends only on shape regularity, $k$ and $\Omega$. Therefore, using Poincare's inequality
\begin{equation*}
\|p\|_{L^2(\Omega)}^2= \int_{\Omega} p \dive v \, dx \le   \|v\|_{H^1(\Omega)} \sup_{ w \in V_h^k, w \neq 0}  \frac{\int_{\Omega} p \dive w \, dx } { \|w \|_{H^1(\Omega)}}.  
\end{equation*}
Hence, 
\begin{equation*}
\|p\|_{L^2(\Omega)} \le  C_2\Big( \frac{1}{\tmin}+1\Big) \sup_{ w \in V_h^k, w \neq 0}  \frac{\int_{\Omega} p \dive w \, dx } { \|w \|_{H^1(\Omega)}}.  
\end{equation*}
The result now follows by letting $\beta =1/\big(C_2 \big(\frac{1}{\tmin}+1\big)\big) $.
\end{proof}

\section{Appendix: Proof of Lemma \ref{lemmasingular}}
 
Let $z \in S_h^2$. 
We first assume that $y$ is an interior vertex.   
In this case,  $\Omega_h(y)= \bigcup_{i=1}^4 T_i$; see Figure \ref{Fig1}.    
Our hypothesis tells us that $v|_{T_i} \in C^1(\overline{T_i})$ for each $i$ and 
that $v \in C(\Omega_h(y))$. 
Let $F(\hat{x})=M \hat{x}+z$ where $M=[t_1, t_2]$ is a $2 \times 2$ matrix.  
Define $\hat{v}(\hat{x})= M^{-1} v(F(\hat{x}))$. 
Note that this is the Piola transformation.  
Let  $R_1, \ldots, R_4$  be the four standard quadrants, and let $L_1, \ldots, L_4$ 
be the four semi-finite lines; see Figure \ref{Fig2}. 
If we let $B$ be a ball with center at the origin and with small enough radius then we 
know that $\hat{v} \in C(B)$ and $\hat{v}|_{B \cap R_i} \in C^1(\overline{B \cap R_i})$ 
for each $i$.  A standard calculation shows that
\begin{equation*}
\dive v(F(\hat{x}))=\dive \hat{v}(\hat{x}) \quad \text{ for } \hat{x} \in  B.
\end{equation*}
Therefore,
\begin{equation*}
A_h^y(\dive  v)= \sum_{i=1}^4 (-1)^{4-i} \dive v|_{T_i}(z)=\sum_{i=1}^4 (-1)^{4-i} \dive \hat{v}|_{R_i}(0).
\end{equation*}

If we let $\hat{v}^i= \hat{v}|_{T_i}$ then we see that
\begin{alignat*}{1}
\sum_{i=1}^4 (-1)^{4-i} \dive \hat{v}|_{R_i}(0)= & -\dive \hat{v}^1(0)+\dive \hat{v}^2(0)-  \dive \hat{v}^3(0)+\dive \hat{v}^4(0)\\
= &-( \partial_1 \hat{v}_1^1+ \partial_2 \hat{v}_2^1)(0)+(\partial_1 \hat{v}_1^2(0) + \partial_2 \hat{v}_2^2(0)) \\
&-( \partial_1 \hat{v}_1^3+ \partial_2 \hat{v}_2^3)(0)+ (\partial_1 \hat{v}_1^4+ \partial_2 \hat{v}_2^2)(0)\\
=& -\partial_1( \hat{v}_1^1 -\hat{v}_1^4)(0)-\partial_1( \hat{v}_1^3 -\hat{v}_1^2)(0)\\
& -\partial_2 ( \hat{v}_2^1-  \hat{v}_2^2)(0)  -\partial_2 ( \hat{v}_2^3-  \hat{v}_2^4)(0) \\
=&0.
\end{alignat*}
In the last step we used that since $\hat{v}$ is continuous $\hat{v}_1^1 -\hat{v}_1^4$ is identically zero on $L_1 \cap B$. Similarly,    $\hat{v}_2^1 -\hat{v}_2^2$ vanishes on $L_2 \cap B$ and so on.
This proves that $A_h^y(\dive  v)=0$.  If, $y \in S_h^2$ is a boundary vertex then we extend $v$ by zero to $\Omega^c$ and apply the previous result. 

\begin{figure}
\vspace{-110pt}
\centerline{\includegraphics[width=12in]{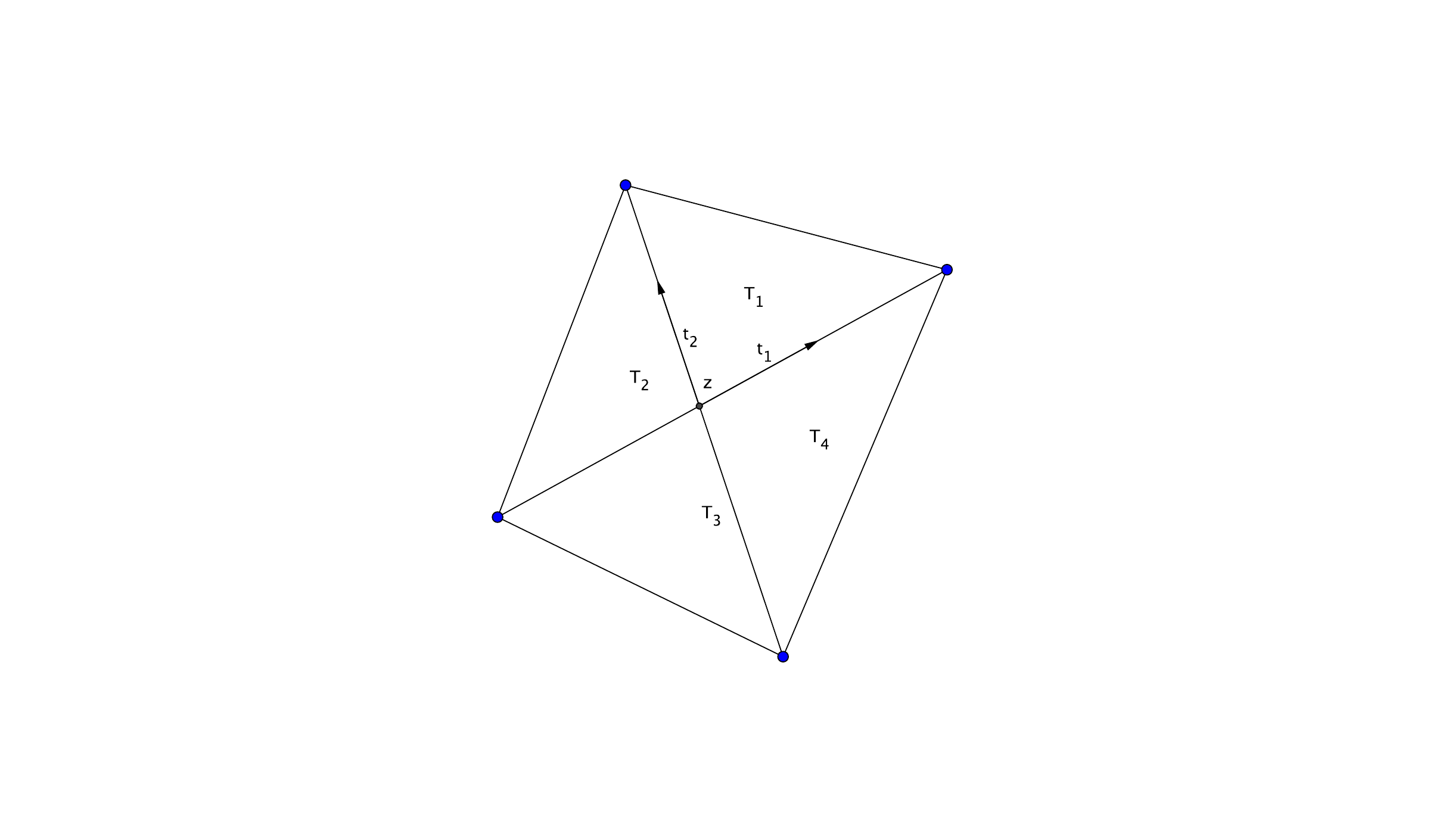}}
\vspace{-110pt}
  \caption{Illustration of $\Omega_h(z)$}
  \label{Fig1}
\end{figure}%
\begin{figure}
 \centerline{\includegraphics[width=5in]{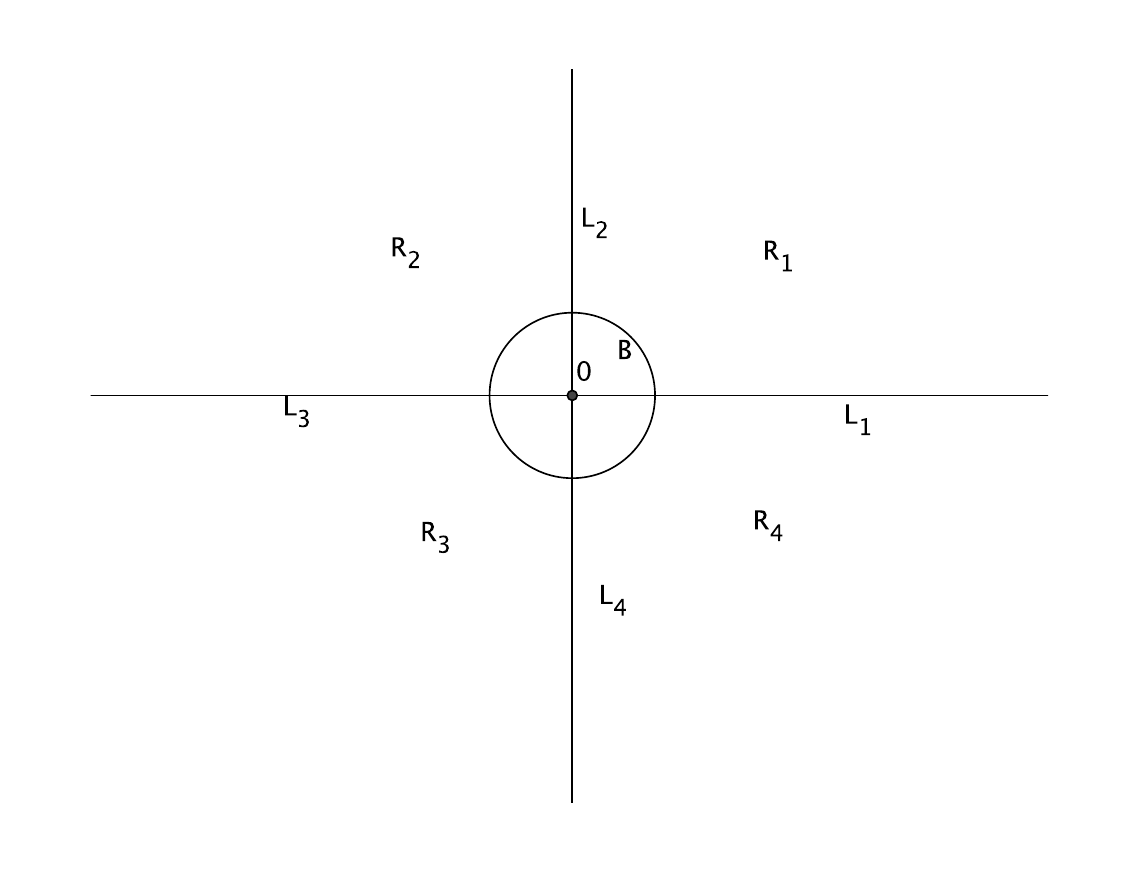}}
\vspace{-30pt}
  \caption{Illustration of four quadrants}
  \label{Fig2}
\end{figure}

\bibliography{Bibliography_BGSS}{}
\bibliographystyle{abbrv}

\end{document}